\documentclass{amsart}

\usepackage{amsmath, amssymb, graphicx, epsfig, verbatim}
\usepackage{epic}

\newtheorem{theorem}{Theorem}[section]
\newtheorem{proposition}[theorem]{Proposition}
\newtheorem{lemma}[theorem]{Lemma}
\newtheorem{corollary}[theorem]{Corollary}

\theoremstyle{definition}

\theoremstyle{remark}

\newtheorem{remark}[theorem]{Remark}

\newcommand{\CS}{\mathcal{CS}}
\newcommand{\CR}{\mathcal R}
\newcommand{\Alb}{\mathrm{Alb}}
\newcommand{\cO}{\mathcal O}
\newcommand{\PP}{\mathbf P}
\newcommand{\QQ}{\mathbf Q}
\newcommand{\ZZ}{\mathbf Z}
\newcommand{\CC}{\mathbf C}

\newcommand{\zetaThree}{\zeta_3}
\newcommand{\NF}{\mathrm{NF}}

\newcommand{\file}[1]{\texttt{#1}}

\begin{document}

\title[A polynomial formula for the Albanese map of the CS surface]
{A Polynomial Formula for the Albanese Map\\
of the Cartwright--Steger Surface}

\author{Zolt\'an Szab\'o}
\address{Department of Mathematics\\
Princeton University\\
 Princeton, NJ, 08544}
\email{szabo@math.princeton.edu}
\date{\today}


\begin{abstract}
The Albanese fibration of the Cartwright--Steger surface is a genus-$19$
Lefschetz fibration over a torus.  We give an explicit formula for it:
in the projective model of Borisov and Yeung, the map is represented by
three homogeneous polynomials of degree $13$ with integer coefficients,
and its target is the elliptic curve $Y^2Z=X^3-3888Z^3$.
A recognition theorem shows that
it suffices to verify one polynomial identity on the surface.  A coefficient
bound allows us to establish the identity over $\QQ$ from computations
modulo primes.  We also determine exact coordinates for the three nodes and their
images on the elliptic curve.
\end{abstract}

\maketitle

\section{Introduction}\label{sec:introduction}

The Cartwright--Steger surface $\CS$ is a compact complex hyperbolic surface
with first Betti number $2$ and Euler characteristic $3$ \cite{CS,CKY}.
It was discovered by Cartwright and Steger \cite{CS} during their
enumeration of the fake projective planes classified by Prasad and Yeung
\cite{PY}.
Its Albanese variety is an elliptic curve, and Cartwright, Koziarz and
Yeung \cite{CKY} proved that the regular fibers of the Albanese map have
genus $19$.
Our aim is to describe this fibration explicitly in the projective
coordinates of Borisov and Yeung \cite{BY}.  We give a polynomial formula
for the map, prove that it is correct, and determine exact coordinates for
its critical points and critical values.

The holomorphic automorphism group of $\CS$ is cyclic of order three;
write $\sigma$ for a generator.  The induced action on the base elliptic
curve has three fixed points \cite{CKY}.  The three $\sigma$-invariant
fibers are smooth and the singular fibers are nodal \cite{KY,Rito}; there
are exactly three of the latter, they are irreducible, and $\sigma$ permutes
them cyclically \cite[Sect.~1.6]{Keum}.  Thus there are two distinguished
triples of fibers: the smooth fibers fixed by $\sigma$ and the nodal fibers
it permutes.

\subsection{Main result}

We use homogeneous coordinates $[U_0:\cdots:U_9]$ in the bicanonical
embedding $\CS\subset\PP^9$ of Borisov and Yeung \cite{BY,Keum}.
Let $X_{13},Y_{13},Z_{13}$ be the degree-$13$ polynomials with integer
coefficients specified in Section~\ref{sec:canonical-triple}. They are a bit too long (around  $10$ pages each) to include here. Instead they are supplied in separate files: {\it X13.txt, Y13.txt, Z13.txt}. 

\begin{theorem}[Degree-$13$ formula for the Albanese map]\label{thm:main}
Let
\begin{equation}\label{eq:target-curve}
 E:\quad Y^2Z=X^3-3888Z^3,
 \qquad O=[0:1:0].
\end{equation}
The rational map
\[
 \alpha=[X_{13}:Y_{13}:Z_{13}]:\CS\dashrightarrow E
\]
extends uniquely to a holomorphic map $\alpha:\CS\to E$, defined over
$\QQ$.  After identifying $\Alb(\CS)$ with $E$, this is the Albanese map,
up to a translation of $E$.  On the open set $Z_{13}\ne0$ its affine
coordinates are
\[
 x=\frac{X_{13}}{Z_{13}},\qquad y=\frac{Y_{13}}{Z_{13}},
 \qquad y^2=x^3-3888.
\]
Set $\zetaThree=e^{2\pi i/3}$.  For the generator $\sigma$ fixed in
Section~\ref{sec:coordinates-symmetry} we have
\begin{equation}\label{eq:target-action}
 \alpha\circ\sigma=\sigma_E\circ\alpha,
 \qquad
 \sigma_E([X:Y:Z])=[\zetaThree^2X:Y:Z].
\end{equation}
Here $\sigma$ and $\sigma_E$ are defined over $\QQ(\zetaThree)$.
The common zero locus of $X_{13},Y_{13},Z_{13}$ on $\CS$ is nonempty.
\end{theorem}

Together with Theorem~\ref{thm:singular-fibers}, this gives an explicit
description of the genus-$19$ Lefschetz fibration over $E$,
with three critical points on distinct fibers.  The singular fibers are
irreducible, so the corresponding vanishing cycles are nonseparating.

The formula gives particularly simple coordinates for the critical points:
\[
 [\,2:1:3:12:0:0:0:-\theta:-\theta:2\theta\,],
 \qquad \theta^3=1116.
\]
They lie on a single $\sigma$-invariant projective line and form one orbit
under $\sigma$.  Section~\ref{sec:singular-fibers} verifies that these points are
nodes of three distinct fibers, computes their images on $E$,
and excludes further critical points by an Euler-characteristic argument.

Let's recall the antiholomorphic action on $\CS$ discovered by Borisov and Yeung in \cite{BY}.
The projective equations and the map are defined over $\QQ$, so complex
conjugation induces an antiholomorphic involution $\iota$ of $\CS$
compatible with conjugation on $E$.  The relation
$\iota\sigma\iota=\sigma^{-1}$ gives an action of the symmetric group
$S_3$ on $\CS$.  Of the three smooth $\sigma$-invariant fibers, only
$F_1=\alpha^{-1}(O)$ is preserved by $\iota$; the other two are exchanged.
Thus the $S_3$-action restricts to the
genus-$19$ fiber $F_1$, with $\sigma$ preserving orientation and $\iota$
reversing it.  Among the three singular fibers, exactly one is preserved
by $\iota$; its node is a transverse crossing of two real branches
(Corollary~\ref{cor:one-real-node}).  For related work on the real quotient of
$\CS$, see \cite{SS}.

The motivation for this paper comes from four-dimensional topology.
The surface is already known in that setting: by the lower bound of
Stipsicz and Yun \cite{SY}, its Albanese fibration has the minimal number of
singular fibers possible for a Lefschetz fibration over the torus,
$N(19,1)=3$ in the notation of \cite{Hamada,SY}, and the surface has served
as a building block for symplectic $4$-manifolds with nonnegative signature
see  Akhmedov  \cite{Akhmedov} and also \cite{ASY}.  For the related question of Lefschetz fibrations with
positive signature over the sphere, see Baykur and Hamada \cite{BH}.
The genus-$19$ Lefschetz fibration over the torus raises several interesting
questions: how does $S_3$ act on the real invariant fiber $F_1$, how do the vanishing
cycles intersect, and what are the monodromies along loops in the Albanese
torus avoiding the critical values?  To study these questions, it is useful
to have an explicit description of the map itself.  The modest goal of the
present paper is to provide such a formula for $\alpha:\CS\to E$.
The topological questions are pursued in subsequent papers.

\subsection{Structure of the proof}

The proof has a geometric part and a polynomial calculation.  By Theorem~\ref{thm:recognition}
it suffices to prove
\[
 Y_{13}^{\,2}Z_{13}-X_{13}^{\,3}+3888Z_{13}^{\,3}=0
 \quad\text{on }\CS.
\]

Given the identity, the map extends to a map $\CS\to E$, which
factors through the Albanese map followed by a covering
between elliptic curves.  An intersection-number estimate shows that the degree of this covering is  one or two
and since both curves have $j$-invariant
$0$ it is easy to show that degree two is impossible.  The proposed map is therefore the Albanese
map (up to the choice of an origin on $E$).

To verify the identity, we reduce the left-hand side to a canonical
remainder modulo a Gr\"obner basis of the defining equations.  We bound
the coefficients of this remainder a priori and compute it modulo enough
primes that their product exceeds the bound; the remainder is therefore
zero over $\QQ$.
Section~\ref{sec:search-strategy} explains how the polynomials were
found, building on Rito's work \cite{Rito}.

\section*{Acknowledgements}

The author first learned of research problems
related to the Cartwright--Steger surface in a lecture by
R.~\.Inan\c{c} Baykur at the {\it Workshop on Exotic 4-manifolds} at Stanford 2023, organized by the Simons Collaboration grant on {\it New Structures in Low-Dimensional Topology}.
A natural question, for example, is
whether this complex surface can be generalized in the smooth or symplectic category.  An ongoing joint
project with Baykur and Andr\'as Stipsicz concerns these and similar
questions; see also \cite{SS}.  I thank \.Inan\c{c} and Andr\'as for their
help and inspiration.  I would also like to thank J\'anos Koll\'ar, Peter
Kronheimer, Tomasz Mrowka and Peter Ozsv\'ath for several very helpful discussions.

The author was partially supported by the Simons Collaboration Grant on
New Structures in Low-Dimensional Topology.

OpenAI's ChatGPT and Anthropic's Claude were used as research-support tools: writing and testing programs,
assisting with
exploratory calculations, and helping with editorial revision.  Their outputs
were not treated as mathematical authority, and the author assumes
full responsibility for the mathematical statements, computations, and any
remaining errors.

\section{The projective model and the order-three symmetry}\label{sec:canonical-model}

This section describes the projective coordinates used to write the Albanese
map.  We first recall the geometry and intersection numbers of the
Cartwright--Steger surface.

Write $K=K_{\CS}$ for the canonical class.

\begin{proposition}\label{prop:published-geometry}
Let $\CS$ be the Cartwright--Steger surface.  Then:
\begin{enumerate}
\item[\textup{(i)}] $\CS$ is a smooth projective complex two-ball quotient, $K$ is ample,
$K^2=9$, $e(\CS)=c_2(\CS)=3$, and $q(\CS)=p_g(\CS)=1$ \cite{CS};
\item[\textup{(ii)}] the bicanonical system embeds $\CS$ in $\PP^9$, with hyperplane class
$H=2K$ \cite{BY,Keum};
\item[\textup{(iii)}] the Albanese variety is an elliptic curve, and the Albanese morphism has
connected general fiber of genus $19$ \cite{CKY};
\item[\textup{(iv)}] the order-three automorphism $\sigma$ has fixed points on $\CS$.
Choosing the image of one of them as the origin of the Albanese variety, the
induced automorphism fixes the origin and has order three \cite{CKY};
\item[\textup{(v)}] every fiber of the Albanese morphism is irreducible and reduced
\cite[Sect.~1.6]{Keum}.
\end{enumerate}
\end{proposition}

Here $q$ and $p_g$ are the dimensions of the spaces of holomorphic one-forms
and two-forms, respectively.  Note that $b_2^+=2p_g+1=3$, $b_1=2q=2$,
$b_2^-=2$ and the signature is $1$.  

For a smooth Albanese fiber $F$, we have $F^2=0$.  Adjunction with
$g(F)=19$, together with $K^2=9$ and $H=2K$, gives
\begin{equation}\label{eq:numerical-intersections}
\begin{gathered}
 F^2=0,\qquad K\cdot F=36,\qquad H\cdot F=72,\\
 K\cdot H=18,\qquad H^2=36.
\end{gathered}
\end{equation}
Thus a fiber has projective degree $72$, while the surface has degree $36$.

\subsection{Coordinates and symmetry}\label{sec:coordinates-symmetry}

We use the homogeneous coordinates $[U_0:\cdots:U_9]$ of Borisov and Yeung
\cite{BY}.  Each coordinate restricts to a section of $H=2K$.
A homogeneous polynomial of degree $d$ therefore restricts to a section
of $dH$. So, unless it vanishes identically, its zero divisor has class $dH$
and projective degree $36d$.  We will use this
relation in Section~\ref{sec:recognition}.

Let $I_{\CS}$ be the ideal generated by the $85$ equations of Borisov and
Yeung: one quadratic and $84$ cubic polynomials.  Let
\[
 S=\QQ[U_0,\ldots,U_9],\qquad \CR=S/I_{\CS}.
\]
We calculate modulo these equations.  The lemma below shows that
they generate every homogeneous polynomial relation on $\CS$.  Thus the
degree-$d$ part $\CR_d$ records the restrictions to $\CS$ of degree-$d$
polynomials, with two polynomials identified when their difference vanishes
on the surface.

The holomorphic automorphism group is cyclic of order three
\cite{CKY,Keum}.  The coordinates diagonalize its chosen generator
\cite[Rem.~5.3]{BY}: $\sigma$ multiplies $U_i$ by $\zetaThree^{w_i}$, where
the weights $w_0,\ldots,w_9$ are
\begin{equation}\label{eq:sigma-weights}
 (0,0,0,0,1,1,1,2,2,2)\pmod3.
\end{equation}

This gives a splitting of $\CR$ according to these weights, also called characters.
Writing $\CR_d^{(j)}$ for the character-$j$ subspace of $\CR_d$ we get
$\CR_d=\CR_d^{(0)}\oplus\CR_d^{(1)}\oplus\CR_d^{(2)}$.

\subsection{Equations and normal forms}\label{sec:groebner}

To calculate modulo the defining equations, we fix a monomial order.
The largest monomial occurring in a nonzero polynomial is its
\emph{leading monomial}, and its coefficient is the \emph{leading
  coefficient}.

We will use a special generating set of the ideal called a Gr\"obner
basis.  In addition we want it to be reduced, so that every polynomial is
\emph{monic}, and no monomial in a basis element is divisible by the
leading monomial of another basis element.  For a fixed ideal and monomial
order, these conditions determine the reduced Gr\"obner basis uniquely,
and Buchberger's algorithm constructs it in finitely many steps \cite{CLO}.

Now given a polynomial that may or may not be in the ideal, we can use
this basis to repeatedly cancel the terms in it that are divisible by a
leading monomial of the basis.  This stops in finitely many steps, and we
get a unique reduced remainder, called the \emph{normal form}.

We use graded reverse lexicographic order with
\begin{equation}\label{eq:the-order}
 U_5>U_4>U_6>U_7>U_9>U_8>U_0>U_1>U_2>U_3.
\end{equation}
For this order the reduced Gr\"obner basis has $83$ elements: one quadratic,
$74$ cubic and eight quartic polynomials.  Denote it by $\mathcal G$.
The basis and a verification over $\QQ$ are supplied in the supplement.

A monomial is \emph{standard} if none of the leading monomials of
$\mathcal G$ divides it.  The standard monomials of degree $d$ form a basis
of $\CR_d$, and the normal form $\NF(f)$ expresses the class of $f$ in this
basis.  Changing the monomial order changes these representatives, but not
the surface or its coordinate ring.  

{\it The quadratic relation.}
For the fixed projective coordinates, the quadratic relation is unique up
to a nonzero scalar.  For the monomial order \eqref{eq:the-order}, its
leading monomial is $U_5^2$.  Rescaling to relatively prime integer
coefficients, we get (written in the monomial order)
\begin{equation}\label{eq:the-quadric}
\begin{gathered}
 19U_5^2-34U_5U_4-330U_4^2
 +166U_5U_6-132U_4U_6+822U_6^2\\
 {}+169U_7U_0+94U_9U_0+26U_8U_0
 +262U_7U_1+1456U_9U_1+428U_8U_1\\
 {}+57U_7U_2+237U_9U_2+137U_8U_2
 +93U_7U_3-39U_9U_3+83U_8U_3=0.
\end{gathered}
\end{equation}

An easy calculation decomposes the quadratic monomials into blocks of
dimensions $(19,18,18)$.  The above quadric has character $2$ (in fact all
$18$ monomials of that character occur in it with nonzero coefficients).
So we see that $\dim\CR_2^{(j)}$ for $j\in\{0,1,2\}$ is given by
$(19,18,17)$.

{\it Dimension counts.}
The function $d\mapsto\dim_{\QQ}\CR_d$ is called the \emph{Hilbert function}.
For sufficiently large $d$ it is given by a polynomial in $d$, the
\emph{Hilbert polynomial}.  Once the leading monomials of $\mathcal G$ are
known, the computation is combinatorial: one counts monomials of each
degree that are not divisible by any of them, with no further polynomial
reduction needed \cite{CLO}.  The \textsc{Magma}~\cite{Magma} program in the supplement
also computes this polynomial and gives
\begin{equation}\label{eq:hilbert-polynomial}
 18d^2-9d+1.
\end{equation}
This agrees with the polynomial predicted by Riemann--Roch from the
intersection numbers above \cite[V.1.6]{Hartshorne}:
\[
 \chi\bigl(\cO_{\CS}(dH)\bigr)
 =\chi(\cO_{\CS})+\tfrac12\,dH\cdot(dH-K)
 =1+\tfrac12(36d^2-18d).
\]
Here $\chi(\cO_{\CS})=1-q+p_g=1$ is the holomorphic Euler characteristic,
distinct from the topological Euler characteristic $e(\CS)=2-2b_1+b_2=3$.

The Hilbert function agrees with \eqref{eq:hilbert-polynomial} in every
degree $d\ge3$. Let's look at degree $d=2$.  There are $55$ quadratic
monomials in ten variables and one independent relation, so
$\dim\CR_2=55-1=54$. For some of the formulas used here and later we use
Kodaira vanishing \cite{BHPV}.  To recall, a divisor or line bundle $A$ is
called ample if the sections of a high enough power of it give an
embedding; then the cohomology groups $H^i(\CS,K+A)$ vanish for $i>0$.
In our case since $2K$ gives an embedding we know that $K$ is ample. It follows
that $3K$ is also
ample, so $h^{i}(\CS,4K)=0$ for $i>0$.
Then with the earlier formulas we have
\[
 h^0(\CS,4K)=\chi\bigl(\cO_{\CS}(4K)\bigr)=55.
\]

As we see below in Lemma~\ref{lem:normal-form-vanishing}, restriction of
quadratic polynomials gives an injective map from $\CR_2$ to
$H^0(\CS,4K)$.  It follows from the above computations that its image is
a codimension one subspace.
For $d\ge3$ we have 
$\CR_d=H^0(\CS,\cO_{\CS}(dH))$: the restriction map is
again injective by the same lemma, and now both spaces have
dimension $18d^2-9d+1$ by \eqref{eq:hilbert-polynomial}, Kodaira vanishing
and Riemann--Roch.  Every section of $dH$ with $d\ge3$ is therefore
represented by a polynomial; this is used repeatedly in
Section~\ref{sec:search-strategy}.

{\it Why reduction detects every relation.}
To use normal forms to detect vanishing on $\CS$, we must verify that
the defining equations generate its full homogeneous ideal.

\begin{lemma}\label{lem:normal-form-vanishing}
The ideal $I_{\CS}$ is the full homogeneous ideal of $\CS$.  A homogeneous
polynomial vanishes identically on $\CS$ if and only if its normal form
with respect to $\mathcal G$ is zero.
\end{lemma}

\begin{proof}
No leading monomial of $\mathcal G$ is divisible by $U_3$.  Thus
multiplication by $U_3$ takes standard monomials to standard monomials
and is injective on $\CR$.

Let $J$ be the full homogeneous ideal of $\CS$.  We have
$I_{\CS}\subseteq J$, and the quotient rings have the same Hilbert
polynomial: that of $S/J$ is $\chi(\cO_{\CS}(dH))$
\cite[Ex.~III.5.2]{Hartshorne}, which agrees with
\eqref{eq:hilbert-polynomial} by Riemann--Roch.
Hence $J_d=(I_{\CS})_d$ for all sufficiently large $d$.  For homogeneous
$f\in J$, choose $k$ sufficiently large that $U_3^kf\in I_{\CS}$.
Injectivity of multiplication by $U_3$ then gives $f\in I_{\CS}$.
Thus $I_{\CS}=J$. 
\end{proof}

\subsection{The fixed points}\label{sec:fixed-points}

The automorphism $\sigma$ has nine fixed points \cite{CKY,Keum}.  Three of
them are the coordinate points $[U_7]$, $[U_8]$, $[U_9]$, where $[U_i]$
denotes the point whose only nonzero coordinate is $U_i$.  The other six
lie in the eigenspace where only $U_0,U_1,U_2,U_3$ can be nonzero, and it
is easy to show from the restricted equations that none of them is real
(\file{fixed\_points.m}).  Section~\ref{sec:recognition} uses only the
existence of a fixed point.

\subsection{The chosen degree-$13$ polynomials}\label{sec:canonical-triple}

We fix the three integral polynomials $X_{13},Y_{13},Z_{13}$ supplied with
this paper.
They have $970$, $976$ and $962$ terms, with integer coefficients of at most
$37$ decimal digits.
Their coefficients, including their relative normalization, remain fixed
throughout.  Independent rescaling of the three entries would in general
change the equation of the target elliptic curve.  An integral polynomial
is \emph{primitive} when the greatest common divisor of its coefficients
is one.

\begin{proposition}\label{prop:interface}
The three polynomials just specified are nonzero, primitive and homogeneous
of degree $13$.  Their characters are $2,0,0$, respectively, and every
monomial occurring in them is standard for $\mathcal G$.  In particular,
none of the three polynomials vanishes identically on $\CS$.
\end{proposition}

\begin{proof}
The degree, character and primitivity assertions follow directly from the
coefficients and exponents.  Each polynomial is nonzero and already in
normal form, so nonvanishing on $\CS$ follows from
Lemma~\ref{lem:normal-form-vanishing}.
\end{proof}

\section{Recognizing the Albanese map}\label{sec:recognition}

This section shows that a degree-$13$ formula with the prescribed symmetry
defines the Albanese map as soon as it satisfies the cubic identity
\eqref{eq:abstract-cubic}.  The verification of that identity for our
chosen polynomials is given in Section~\ref{sec:height-bound}.

\begin{theorem}[Degree-$13$ Albanese recognition]\label{thm:recognition}
Let $A,B,C\in\QQ[U_0,\ldots,U_9]$ be homogeneous of degree $13$, with
characters $2,0,0$, respectively, for the action \eqref{eq:sigma-weights}.
Assume that none vanishes identically on $\CS$, and that
\begin{equation}\label{eq:abstract-cubic}
 B^2C=A^3-3888C^3
\end{equation}
on $\CS$.  Then $[A:B:C]$ extends uniquely to a nonconstant algebraic
morphism $f:\CS\to E$, defined over $\QQ$ and equivariant for the
order-three actions.  After identifying $\Alb(\CS)$ with $E$ over $\CC$,
it is the Albanese map up to a translation of $E$.
\end{theorem}

\begin{proof}
\textbf{Extending the map.}
The cubic identity \eqref{eq:abstract-cubic} defines a rational map
$f_0=[A:B:C]:\CS\dashrightarrow E$.  Since the nonzero rational function
$A/C$ has character two, $f_0$ is nonconstant.  The character conditions
also give
\begin{equation}\label{eq:rational-equivariance}
 f_0\circ\sigma=\sigma_E\circ f_0
\end{equation}
on the domain of the formula.

A rational map from a smooth variety to an abelian variety extends
to an algebraic morphism over the same ground field
\cite[Thm.~3.1]{MilneAV}.  Applied over $\QQ$, this gives a morphism
$f:\CS\to E$ defined over $\QQ$, unique since it is determined on a
dense open subset.  Both sides of \eqref{eq:rational-equivariance} are
morphisms over $\QQ(\zetaThree)$ agreeing on a dense open subset, so
\eqref{eq:rational-equivariance} holds on all of $\CS$.

For the rest of the proof, we work over $\CC$.

\medskip
\noindent\textbf{Determining the covering degree.}
Choose a $\sigma$-fixed point $\xi$ and normalize the Albanese map
$a:\CS\to\mathcal A=\Alb(\CS)$ by $a(\xi)=0$.  By the universal property
of the Albanese map, there is a homomorphism $h:\mathcal A\to E$ with
$f(p)=h(a(p))+e$, where $e=f(\xi)$.  Since $f$ is nonconstant, $h$ is
nonzero, hence an \emph{isogeny}---an unbranched holomorphic covering---of
some degree $n\ge1$.  A regular fiber of $f$ is the disjoint union of $n$
Albanese fibers.  We must prove that $n=1$.

Choose a generic line $\ell:\lambda X+\mu Y+\nu Z=0$ meeting $E$ in three
distinct regular values of $f$.  Their inverse images consist of $3n$
disjoint Albanese fibers.  The section
\[
 s_\ell=\lambda A+\mu B+\nu C
\]
is nonzero because $f$ is nonconstant, and its zero divisor contains
these fibers, each with multiplicity at least one.  Any remaining
components, including common components of $A,B,C$, contribute
nonnegatively to its intersection with $H$.  Therefore, by
\eqref{eq:numerical-intersections},
\[
 3n(H\cdot F)\le H\cdot\operatorname{div}(s_\ell)=13H^2,
 \qquad 216n\le468.
\]
Hence $n\in\{1,2\}$.

Both $\mathcal A$ and $E$ have an automorphism of order three fixing the
origin, by Proposition~\ref{prop:published-geometry}\textup{(iv)} and
\eqref{eq:target-action}, so both have $j$-invariant $0$
\cite[Thm.~III.10.1]{Silverman} and are isomorphic over $\CC$ to
\[
 \CC/\Lambda,\qquad \Lambda=\ZZ+\ZZ\zetaThree.
\]
In these coordinates $h$ is induced by multiplication by
$\beta=r+s\zetaThree$ with $r,s\in\ZZ$, and the covering
degree is
\[
  n= =r^2-rs+s^2\equiv(r+s)^2\pmod3.
\]
Thus $n\ne2$ and it follows that $h$ is an isomorphism, and $f$ is the Albanese
map up to translation.
\end{proof}

{\it The common zeros of the formula.}
Fix $p\in\CS$ with $f(p)$ in the chart $X\ne0$, and 
write $f=[1:v:w]$.  For local representatives of the sections $A,B,C$,
the identity $(A,B,C)=A\,(1,v,w)$ holds on the dense open set where the
formula is defined, hence on the whole neighborhood.  Thus the common
zero locus is locally the divisor of $A$, and similarly of $B$ or $C$ in
the other target charts.  These local factors define an effective
divisor $\Delta$; after removing it, the three sections have no common
zero.

Since $n=1$, the zero divisor of a generic linear combination
$\lambda A+\mu B+\nu C$ consists of $\Delta$ and three Albanese fibers.
Thus
\[
 13H\equiv\Delta+3F
\]
numerically, and
\[
 H\cdot\Delta=13H^2-3H\cdot F=468-216=252.
\]
In particular, the common zero locus is nonempty, as asserted in
Theorem~\ref{thm:main}.

\section{Proving the cubic identity from reductions modulo primes}\label{sec:height-bound}

Set
\begin{equation}\label{eq:F39}
 \Phi_{39}=Y_{13}^{\,2}Z_{13}-X_{13}^{\,3}+3888Z_{13}^{\,3},
\end{equation}
a form of degree $39$.
By the recognition theorem, it remains to show that $\Phi_{39}$ vanishes on
$\CS$.  By Lemma~\ref{lem:normal-form-vanishing}, this is equivalent to
$R=\NF(\Phi_{39})=0$, where the normal form is taken with respect to the
$83$-element basis $\mathcal G$ of Section~\ref{sec:groebner}.  We bound the
coefficients of $R$ after clearing denominators, and calculate its
reductions modulo primes, without directly computing $R$ over $\QQ$.

\subsection{Why computations modulo primes suffice}\label{sec:membership-criterion}

Call a prime $p$ \emph{admissible} if it divides none of the denominators
of the coefficients in $\mathcal G$.  Since $\mathcal G$ is monic, for
admissible $p$ every reduction by $\mathcal G$ over $\QQ$ of a polynomial
with denominators prime to $p$ makes sense modulo $p$.  Using this to check 
 Buchberger's criterion \cite{CLO} it follows that
$\mathcal G\bmod p$ is a Gr\"obner basis of the ideal it generates, with
the same leading monomials (cf.~\cite{Arnold}).  We have
\[
 R\bmod p=\NF_{\mathcal G\bmod p}(\Phi_{39}\bmod p).
\]

We will obtain a positive integer $D$ for which
\[
 D R=\sum_\mu N_\mu s_\mu,\qquad N_\mu\in\ZZ,
\]
where the $s_\mu$ are standard monomials.  If $R$ vanishes modulo distinct
admissible primes $p_i$ that do not divide $D$, then every $N_\mu$ is divisible
by $\prod_i p_i$.  Consequently, $R=0$ once this product exceeds a bound for
the absolute values of the $N_\mu$.  This is a standard use of modular
arithmetic in polynomial calculations see
\cite[Ch.~5]{vonZurGathenGerhard} and \cite{Arnold}.

\subsection{Bounding the remainder}\label{sec:deterministic-rule}

For $f=\sum_m c_m m\in\QQ[U_0,\ldots,U_9]$, written in distinct
monomials, we use $\|f\|_1=\sum_m|c_m|$.
It is straightforward to compute that

\[ \|\Phi_{39}\|_1 <2^{387}. \]

Let $\mathcal M$ consist of all monomials obtained by multiplying terms in
the three products defining $\Phi_{39}$.  This set can be determined from
the monomials in $X_{13},Y_{13},Z_{13}$, without calculating the coefficients
of $\Phi_{39}$.  Write
\[
 \Phi_{39}=\sum_{m\in\mathcal M}a_m m,\qquad a_m\in\ZZ,
\]
allowing $a_m=0$ when terms cancel.  By linearity of normal form,
\[
 R=\sum_{m\in\mathcal M}a_m\NF(m),\qquad
 \|R\|_1\le\|\Phi_{39}\|_1
       \max_{m\in\mathcal M}\|\NF(m)\|_1.
\]
A common denominator for these monomial normal forms also clears the
denominators of $R$, since the $a_m$ are integers.

{\it Recursive bounds.}
Write a basis polynomial as
\[
 g=m_0+\sum_{\nu=1}^{t}\kappa_\nu m_\nu.
\]
Here $m_0$ is the leading monomial and the $\kappa_\nu$ are the nonzero
remaining coefficients.  If $m_0\mid m$, then
\begin{equation}\label{eq:nf-recurrence}
 \NF(m)=-\sum_{\nu=1}^{t}\kappa_\nu\NF(m_\nu'),
 \qquad m_\nu'=\frac{m}{m_0}\,m_\nu,
\end{equation}
and each $m_\nu'$ is smaller than $m$ in the monomial order.
For each monomial $m$, we compute integers $b(m)$ and $e_q(m)$
such that
\[
 \|\NF(m)\|_1\le2^{b(m)},\qquad
 D(m)\NF(m)\in\ZZ[U_0,\ldots,U_9],
 \qquad D(m)=\prod_q q^{e_q(m)}.
\]
Here $q$ ranges over the primes occurring in the denominators of
$\mathcal G$.  For standard monomials, take $b(m)=e_q(m)=0$.
For a nonstandard monomial, the reduction
\eqref{eq:nf-recurrence} gives the recurrences
\begin{equation}\label{eq:nf-rec}
\begin{aligned}
 e_q(m)&=\max_\nu\bigl(
   v_q(\operatorname{den}(\kappa_\nu))+e_q(m_\nu')\bigr),\\
 b(m)&=\lceil\log_2t\rceil+
   \max_\nu\bigl(\lceil\log_2|\kappa_\nu|\rceil+b(m_\nu')\bigr).
\end{aligned}
\end{equation}
Here $\operatorname{den}(\kappa_\nu)$ is its positive denominator
in lowest terms, and $v_q(d)$ is the exponent of $q$ in $d$.
The first recurrence gives a common denominator for the terms in
the sum; the second bounds their coefficient sum.  Since every $m_\nu'$ is smaller than $m$,
these estimates follow by induction.

The bounds depend on the choice of reduction, although the normal
form does not.  For the algorithm we give a rule that works reasonably well.  Given a monomial
$m$ that is  not  standard  there is at least one $g\in \mathcal{G}$ whose leading monomial divides
it.  If there is more than one of these, then  choose the $g$ by requiring that after replacing $m$ we get an expression 
where the  largest resulting monomial is as  small as possible. In case there  is a tie then 
choose among these  $g$ the one that gives  fewer resulting monomials.
If among these there is still a tie then choose
the $g$ that  appears first in the list  generators.

The program \file{nf\_bound83.cpp} computes each encountered monomial's
bounds once and stores only $b(m)$ and the denominator exponents $e_q(m)$,
rather than its normal form.  The set $\mathcal M$ has about $2.6$ million
elements, and the recursion visits about $25$ million monomials.

\subsection{The bounds and the cubic identity}

Put
\[
 B=\max_{m\in\mathcal M}b(m),\qquad
 d_{39}
   =\prod_q q^{\,\max_{m\in\mathcal M}e_q(m)}.
\]
The calculation gives the following bounds.

\begin{proposition}[Monomial bounds]\label{prop:monomial-height}
For every $m\in\mathcal M$,
\[
 \|\NF(m)\|_1\le2^{B},\quad B=852,\qquad
 d_{39}\NF(m)\in\ZZ[U_0,\ldots,U_9],
\]
where
\[
 d_{39}=2^{242}3^{116}5^{61}7^{55}11^{71}13^{71}17^{32}19^{62}23^{51}97^{35}103^{45}137^{71}
  \]

\end{proposition}

The second calculation reduces the polynomial modulo the chosen primes.

\begin{proposition}[Prime reductions]\label{prop:prime-reductions}
Let $p_1,\ldots,p_{68}$ be the $68$ largest primes below $2^{61}$ that are
congruent to $1$ modulo $3$.  They are
admissible, none divides $d_{39}$, and
\[
 \NF_{\mathcal G\bmod p_i}(\Phi_{39}\bmod p_i)=0
 \qquad(1\le i\le68).
\]
Their product satisfies
\begin{equation}\label{eq:prime-product-bound}
 T=\prod_{i=1}^{68}p_i>2^{4130}.
\end{equation}
\end{proposition}

\begin{remark}[Calculations]\label{rem:two-programs}
  As we have seen  we need two kinds of programs. The first computes  the bound and the second checks
  the cubic identity over primes. For the interested reader the programs are also provided in the supplement.
  The first one is the file \file{nf\_bound83.cpp} and that is already discussed above. The $68$ primes are in
  \file{primes.txt}.
  For the cubic identity over a prime, there is a \textsc{Magma} file \file{identities.m}.
  It is standard, however it is memory intensive using around
  $140$ GB and around twenty minutes per prime. The alternative is the \file{reduce\_cubic83.cpp} that runs on a normal laptop
  and is faster (around three minutes per prime).
 \end{remark}

\begin{theorem}[The cubic identity]\label{thm:cubic-identity}
In the coordinate ring $\CR$ over $\QQ$,
\begin{equation}\label{eq:cubic-identity}
 Y_{13}^{\,2}Z_{13}=X_{13}^{\,3}-3888Z_{13}^{\,3}.
\end{equation}
\end{theorem}

\begin{proof}
We see that $d39\cdot R$ has only integer coefficient. Each coefficient  $N_\mu$
satisfies
\[
 |N_\mu|\leq d_{39}\,\|\Phi_{39}\|_1
                  \max_{m\in\mathcal M}\|\NF(m)\|_1
           < 2^{4130}.
\]
By Proposition~\ref{prop:prime-reductions} and the compatibility of reduction in
Section~\ref{sec:membership-criterion}, all these integers are divisible by
$T >2^{4130}$ and hence vanish.  Thus $R=\NF(\Phi_{39})=0$, proving
the identity.
\end{proof}

\begin{proof}[Proof of Theorem~\ref{thm:main}]
Proposition~\ref{prop:interface} gives the degree, symmetry and nonvanishing
conditions for the three polynomials.  The cubic identity verifies the
remaining hypothesis of Theorem~\ref{thm:recognition}, which gives the Albanese
map and its equivariance.  The common-divisor calculation in
Section~\ref{sec:recognition} gives the asserted nonempty common zero locus.
\end{proof}

Before turning to how the polynomials were found, we use the formula to
locate the singular fibers.

\section{The three nodes and their critical values}\label{sec:singular-fibers}

We now determine the critical points of $\alpha$ and their images on $E$.

\begin{theorem}[The three nodes]\label{thm:singular-fibers}
The critical points of $\alpha$ are
\begin{equation}\label{eq:node-points}
 P_\theta=[\,2:1:3:12:0:0:0:-\theta:-\theta:2\theta\,],
 \qquad \theta^3=1116.
\end{equation}
Their critical values are
\begin{equation}\label{eq:critical-target}
\begin{gathered}
 \alpha(P_\theta)=[\,q_{\mathrm{crit}}\theta:y_{\mathrm{crit}}:1\,],\\
 q_{\mathrm{crit}}=\frac{6267859}{1459^2},
 \qquad y_{\mathrm{crit}}=\frac{487135434066}{1459^3}.
\end{gathered}
\end{equation}
The three points lie on distinct fibers, and each is an ordinary node of
its fiber.  These are the only singular fibers of $\alpha$.
\end{theorem}

\subsection{The three points on a line}\label{sec:nodal-line}

Consider the projective line $L$ parametrized by
\begin{equation}\label{eq:nodal-line}
 P(t)=[\,2:1:3:12:0:0:0:-t:-t:2t\,],\qquad t\in\CC,
\end{equation}
together with its point at infinity
$B_\infty=[\,0:0:0:0:0:0:0:-1:-1:2\,]$.
The line $L$ is $\sigma$-invariant, with
\[
 \sigma(P(t))=P(\zetaThree^2t).
\]

Substitution into the $85$ surface equations gives $78$ zero polynomials
and seven nonzero rational multiples of $t^3-1116$.  The point $B_\infty$
does not lie on the surface.  Since $t^3-1116$ has three distinct roots,
$L\cap\CS$ consists of precisely the three points \eqref{eq:node-points},
which $\sigma$ permutes cyclically.

\subsection{The local singularities}\label{sec:node-criticality}

By equivariance, it suffices to work at the real point $P=P_\theta$,
where $\theta=\sqrt[3]{1116}>0$.  Substitution gives $Z_{13}(P)\ne0$
and the image in \eqref{eq:critical-target}.  Since
$x(\alpha(P))=q_{\mathrm{crit}}\theta\ne0$, the function $y-y_{\mathrm{crit}}$
is a local coordinate on $E$.  Its pullback is
\[
 g=\frac{Y_{13}}{Z_{13}}-y_{\mathrm{crit}}.
\]
In the affine chart $U_0=1$, where $U_8(P)=-\theta/2$ and $U_9(P)=\theta$,
the Jacobian at $P$ has rank seven, with a nonzero $7\times7$ minor in the
$U_1,\ldots,U_7$ columns.  Thus
\[
  u=U_8+\theta/2,\qquad v=U_9-\theta
\]
are local holomorphic coordinates centered at $P$.  The supplement
(script \file{node\_local.m}) verifies, exactly over $\QQ(\theta)$, that
$dg$ vanishes on $T_P\CS$ and that minus the determinant of the Hessian of
$g|_{\CS}$ at $P$ in these coordinates is a nonzero square.  The Hessian is
\begin{equation}\label{eq:node-hessian}
\begin{gathered}
 \operatorname{Hess}_{(u,v)}(g|_{\CS})(P)
 =-\frac{4q_{\mathrm{crit}}^{\,2}\theta}{899}
   \begin{pmatrix}164&113\\113&72\end{pmatrix},\\
 \det\operatorname{Hess}_{(u,v)}(g|_{\CS})(P)
 =-\left(\frac{4q_{\mathrm{crit}}^{\,2}\theta}{29}\right)^2\ne0.
\end{gathered}
\end{equation}
The holomorphic Morse lemma therefore gives local holomorphic
coordinates centered at $P$ in which $g=z_1z_2$, so the fiber has an
ordinary node.

\begin{proof}[Completion of the proof of
Theorem~\ref{thm:singular-fibers}]
Let $F$ be a regular fiber.  Since every fiber is reduced by
Proposition~\ref{prop:published-geometry}\textup{(v)}, the
Euler-characteristic formula expresses
\[
  e(\CS)-e(E)e(F)=3
\]
as a sum of positive local contributions from the singular points
of the fibers.  Each ordinary node contributes one. The three nodes already found have distinct
images by \eqref{eq:critical-target} and contribute three in total.
There are therefore no further critical points or singular fibers.
\end{proof}

\begin{corollary}[One real node]\label{cor:one-real-node}
Exactly one of the three singular fibers is preserved by complex
conjugation; the other two are exchanged.  The real fiber has its node
at $P_\theta$ with $\theta=\sqrt[3]{1116}>0$, and its real locus near the
node consists of two branches crossing transversely.
\end{corollary}

\begin{proof}
Since $t^3-1116$ has exactly one real root, formula
\eqref{eq:critical-target} gives one real critical value and a conjugate
pair.  At the real point, $u,v$ are real local coordinates and the
determinant in \eqref{eq:node-hessian} is negative.  The real Morse
lemma gives $g=r^2-s^2$, so the real fiber has two transverse branches.
\end{proof}

The points were originally found by a search over some primes on
$\CS/\langle\sigma\rangle$, where they have a common image.
Reconstructing this quotient point over $\QQ$ and lifting it to $\CS$
yielded the coordinates in \eqref{eq:node-points}.  The verification
above is much simpler.

\section{How the formula was found}\label{sec:search-strategy}

The proof in Sections~\ref{sec:canonical-model}--\ref{sec:height-bound}
of course does not depend on how the polynomials were found.  But it
might be interesting to see the method.  This section gives a short
account following Rito's construction of the invariant fiber $F_1$ and
the degree-$9$ pencil \cite[Sect.~4.1]{Rito}.

Write
$I=I_{\CS}$ and $\CR=S/I$, let $\mathcal A$ be the elliptic base, and
let $p_1\in\mathcal A$ be the image of $F_1$, taken as the origin.

Every step computes a space of sections
$H^0(\CS,\cO_{\CS}(dH-D))$ for some divisor $D$, as a
subspace of $\CR_d=H^0(\CS,\cO_{\CS}(dH))$ (Section~\ref{sec:groebner},
$d\ge3$).  Vanishing along $D$ is imposed by a multiplication equation.
Let $s,t\in\CR$ be nonzero homogeneous sections with known divisors, and
suppose that $e=d+\deg s-\deg t\ge3$.  The sections $f\in\CR_d$ for which
\[
 fs=gt\quad\text{in }\CR
\]
has a solution $g\in\CR_e$ are exactly those satisfying
\[
 \operatorname{div}(f)+\operatorname{div}(s)\ge\operatorname{div}(t).
\]
Indeed, this inequality makes $fs/t$ a global section of $eH$, and
Section~\ref{sec:groebner} gives its polynomial representative.  In the
standard-monomial bases this is a linear system for the coefficients of
$f$ and $g$.  All these systems can be solved modulo primes, and then one
can try to lift them to $\QQ$
\cite[Sect.~2]{Rito}, \cite[Ch.~5]{vonZurGathenGerhard}.

Most of the steps are in Rito's work:
Borisov identified the hyperplane section $U_1=0$ as the union of two
curves $E_1,E_2$ of degree $18$ and arithmetic genus $8$
(see \cite[Sect.~4.1]{Rito}, \cite{Keum}), so $H\equiv E_1+E_2$.
Rito gives the ideal for both of them and shows that $F_1$ is the unique
element of $|5E_2-E_1|$; equivalently, there is a quintic $q_5$, unique
up to a scalar, with
\[
 \operatorname{div}(q_5)=6E_1+F_1.
\]

The quintics vanishing on $F_1$ are the sections $f\in\CR_5$ for which
\[
 fU_1^6=q_5g,\qquad g\in\CR_6,
\]
is solvable.  Let $W_5$ denote the corresponding space.

Rito also constructs over a prime field two degree-$9$ forms $P_9$ and
$Q_9$ with
$\operatorname{div}(Q_9)=2F_1+C_0$ and $\operatorname{div}(P_9)=C_0+F_2+F_3$
for some divisor $C_0$, so their ratio $x_0=P_9/Q_9$ gives a pencil
through $2F_1$ and $F_2+F_3$.  On the base this is the expected function:
the functions on $\mathcal A$ with at most a double pole at $p_1$ and no
other poles are spanned by $1$ and a function $x_0$ of character two;
since $x_0$ vanishes at the other two fixed points $p_2,p_3$ of $\sigma$,
its divisor is $p_2+p_3-2p_1$, and the corresponding fibers are $F_1$,
$F_2$ and $F_3$.
A pair of degree-$9$ forms over $\QQ$ is given in the files
{\it P9.txt, Q9.txt}.

\subsection{The second coordinate and the cubic}\label{sec:sextic-twist-paper1}

The functions with at most a triple pole at $p_1$ form a
three-dimensional space, and a third function $y_0$ in it,
independent of $1$ and $x_0$, has pole order exactly three.  A common
denominator for $x_0$ and $y_0$ must vanish at least three times
along $F_1$.  Choose a nonzero invariant quintic
\[
 f_0\in W_5\cap\CR_5^{(0)},\qquad
 \operatorname{div}(f_0)=F_1+R_0.
\]
Then
\[
 Z_0=Q_9f_0,\qquad X_0=P_9f_0
\]
have characters zero and two, respectively, and
\[
 \operatorname{div}(Z_0)=3F_1+\Delta_{14},\qquad
 \Delta_{14}=C_0+R_0.
\]
Here $H\cdot\Delta_{14}=288$.
The degree-$14$ sections vanishing along $\Delta_{14}$ are  pullbacks
of the sections of $\cO_{\mathcal A}(3p_1)$; they form a
three-dimensional space containing $Z_0$ and $X_0$, and an invariant
third element $Y_0$ gives $y_0=Y_0/Z_0$.

An additional issue is to find the shape of the elliptic curve over $\QQ$.
Note that the $3$-fold symmetry implies that the $j$-invariant is $0$.
Over $\QQ$ the computation can be shortened with some guesswork.
Work in the family
\[
 y^2=x^3+d
\]
with $d\in \ZZ$ and assume that the only possible bad primes are $2$
and $3$.  Then there are only a few $d$ to check.  Eventually
\[
 d = -3888= -2^4\cdot3^5
\]
works.

The degree-$13$ representatives were found from the linear system
\begin{equation}\label{eq:degree13-system}
 X'Q_9-Z'P_9=0\quad\text{in }\CR_{22},\qquad
 Y'Z_0-Z'Y_0=0\quad\text{in }\CR_{27}.
\end{equation}
Every nonzero triple in the kernel represents the same map, so the
remaining freedom was used to make the coefficients small.  The first
triples found were rather large and had coefficients with more than $700$
decimal digits.  In the hope of getting smaller formulas one can try a few
obvious ideas.  These include changing the monomial order for the
Gr\"obner basis, trying to lower the degree of the polynomials and
searching in the family of triples.  All of these help to some degree,
and the final (degree $13$) polynomials have coefficients using at most
$37$ decimal digits.

The pair $P_9,Q_9$ and the triple $X_{13},Y_{13},Z_{13}$ are compatible
in the following sense.

\begin{remark}[Degree-$9$/degree-$13$ ratio identity]\label{rem:degree9-degree13-ratio}
In the coordinate ring $\CR$ over $\QQ$,
\begin{equation}\label{eq:degree9-degree13-ratio}
 X_{13}Q_9-Z_{13}P_9=0.
\end{equation}
Consequently, as rational functions on $\CS$,
$P_9/Q_9=X_{13}/Z_{13}=x$.  This can be checked along the lines of
Section~\ref{sec:height-bound}.
\end{remark}

\end{document}